\documentclass[11pt]{article}

\usepackage[margin=1.15in]{geometry}
\usepackage{amsmath,amssymb,amsthm,mathtools}
\usepackage[colorlinks=true,citecolor=blue,linkcolor=blue,urlcolor=blue]{hyperref}
\usepackage{microtype}
\usepackage{setspace}
\newtheorem{theorem}{Theorem}[section]

\newtheorem{corollary}[theorem]{Corollary}
\newtheorem{lemma}[theorem]{Lemma}
\newtheorem{question}[theorem]{Question}

\theoremstyle{remark}

\newtheorem{fact}[theorem]{Fact}

\newcommand{\F}{\mathbb F}
\DeclareMathOperator{\Tr}{Tr}
\DeclareMathOperator{\Nm}{Nm}

\title{Polynomially superlinear growth of set-coloring Ramsey numbers\footnote{Center for Discrete Mathematics, Fuzhou University,
Fuzhou 350108, P.~R.~China. Email: {\tt linqizhong@fzu.edu.cn, 1539166573@qq.com}. Supported in part by National Key R\&D Program of China (No. 2023YFA1010202) and NSFC (No.\ 12571361).}}
\hypersetup{
  pdftitle={Polynomially superlinear growth of set-coloring Ramsey numbers},
  pdfsubject={Set-coloring Ramsey numbers near the zero-rate threshold},
  pdfkeywords={set-coloring Ramsey numbers, error-correcting codes, ovoid codes}
}

\begin{document}

\author{
Qizhong Lin \qquad and \qquad
Lin Niu}
\date{}

\maketitle

\begin{abstract}
The set-coloring Ramsey number $R(k;r,s)$ is the least $N$
such that every assignment of an $s$-element subset of $[r]$ to each edge of $K_N$ yields a copy of $K_k$
whose edges share a common color.
For every fixed prime power $q$, we construct infinitely many
positive integer triples $(r,j,s)$ with
$j\sim(q-1)^{-2/3}r^{1/3}$ and $s=(1-1/q)(r-j)$
such that $R(q+1;r,s)=\Theta_q(r^{4/3})$.
For $q=3$, this answers in the affirmative a question of Conlon,
Fox, Pham and Zhao, showing that polynomially superlinear growth for $R(4;r,2(r-j)/3)$ already occurs at the scale
\(j=\Theta(r^{1/3})\). Moreover, along the same sequence, the maximum size of a $q$-ary code
of length $r$ and minimum Hamming distance at least $s$ is $(1+o(1))(q-1)^{4/3}r^{4/3}$.

\medskip
\textbf{Keywords:} Set-coloring Ramsey numbers, error-correcting
codes, ovoid codes.

\end{abstract}

\section{Introduction}

For positive integers $k,r,s$ with $r>s$, the set-coloring Ramsey
number $R(k;r,s)$ is the least $N$ such that every assignment of an
$s$-element subset of $[r]=\{1,\ldots,r\}$ to each edge of $K_N$
yields a copy of $K_k$ whose edges share a common color.
The case $s=1$ is the classical $r$-color Ramsey number.
The other extreme, $s=r-1$, was studied by Erd\H{o}s, Hajnal and
Rado~\cite{EHR}; Erd\H{o}s and Szemer\'edi~\cite{ESzemeredi} later
proved $2^{ck/r}\le R(k;r,r-1)\le2^{c'k\log r/r}$ for positive
absolute constants $c,c'$, with $k$ sufficiently large in terms of $r$.
Arag\~ao, Collares, Marciano, Martins and Morris~\cite{Aragao}
introduced a random set-coloring construction that improves the
general lower bounds, notably when \(s/r\) is close to one.
For further work on set-colorings and related Ramsey problems, see
\cite{AEGM,BucicKhamseh,BustamanteStein,HeMao2025,PikhurkoStadenYilma,XuShaoSuLi,ZhaoXieXuChen}
and the references therein.

For fixed \(k\ge3\), the behavior changes at \(s/r=1-1/(k-1)\).
Conlon, Fox, He, Mubayi, Suk and Verstra\"ete~\cite{CFHMSV}
proved that for all sufficiently large \(r\),
\[
R(k;r,s)\ge2^{cr}
\]
whenever \(s/r\le1-1/(k-1)-\varepsilon\), where \(\varepsilon>0\)
is fixed and \(c=c(k,\varepsilon)>0\).
They also showed that \(R(k;r,s)\) is bounded in terms of \(k\)
and \(\varepsilon\) when \(s/r\ge1-1/(k-1)+\varepsilon\).

We focus on the onset of polynomially superlinear growth just below
this threshold. Write $k=q+1$, so that the threshold becomes
$s/r=1-1/q$.

The value $1-1/q$ is also the classical zero-rate threshold
for the relative minimum distance of $q$-ary codes. A $q$-ary code of
length $r$ is a subset of $[q]^r$, and its minimum Hamming distance
is the least number of coordinates in which two distinct codewords
differ. Let $A_q(r,d)$ denote the maximum size of such a code with
minimum distance at least $d$, with the convention
$A_q(r,d)=A_q(r,\lceil d\rceil)$ for real $d$.
The Plotkin bound, originally proved for binary
codes~\cite{Plotkin}, and its $q$-ary extension rule out positive
asymptotic rate when the minimum distance is at least
$(1-1/q)r$ (see~\cite{CFPZ}).

The connection with set-coloring Ramsey numbers is made explicit
by the inequality
\begin{equation}\label{eq:code-Ramsey-lower}
R(q+1;r,s)>A_q(r,s),
\end{equation}
observed in~\cite{CFHMSV}.
Indeed, take the codewords as vertices and assign to each edge $xy$
an $s$-element subset of the coordinates in which $x$ and $y$ differ.
A copy of $K_{q+1}$ with common color $c$ would then have $q+1$
pairwise distinct entries in coordinate $c$, which is impossible.

Conlon, Fox, Pham and Zhao~\cite{CFPZ} proved an approximate
converse: for every integer \(q\ge2\) and \(\varepsilon>0\), there is
\(c=c(q,\varepsilon)>0\) such that, for positive integers \(r,s\)
with \(s\le(1-1/q)r\),
\begin{equation}\label{eq:CFPZ-upper}
R(q+1;r,s)\le
\max\bigl\{(1+\varepsilon)A_q(r,s-c\Delta),\,\varepsilon s\bigr\},
\qquad \Delta=(1-1/q)r-s+1.
\end{equation}
Thus coding upper bounds transfer at the cost of reducing the
minimum distance by \(c\Delta\). Related results in~\cite{CFPZ}
give \(R(3;2s,s)=(4+o(1))s\) and refine bounds of
Pang, Mahdavifar and Pradhan~\cite{Pang} for binary codes whose
minimum distance is \(r/2-\Theta(\sqrt r)\).

To examine the transition near the threshold, write
\(s=(1-1/q)(r-j)\), where \(j\ge0\) and \(s\) is a positive integer.
Then \((1-1/q)j\) is the deficit, and we ask how large \(j\) must
be for \(R(q+1;r,s)\) to grow polynomially faster than \(r\).

Balla~\cite{Balla} proved that, for fixed \(q\ge2\),
\begin{equation}\label{eq:Balla-code}
A_q\bigl(r,(1-1/q)(r-j)\bigr)
\le(2+o(1))(q-1)r
\qquad\text{if }j=o(r^{1/3}).
\end{equation}
His proof uses a bound on nearly orthogonal unit vectors and
the simplex representation of \(q\)-ary codes. For \(q=2\), this bound resolves
a conjecture of Tiet\"av\"ainen~\cite{Tietavainen}.
Combining \eqref{eq:Balla-code} with \eqref{eq:CFPZ-upper},
Balla also obtained
\begin{equation}\label{eq:Balla-Ramsey}
R\bigl(q+1;r,(1-1/q)(r-j)\bigr)
\le(2+o(1))(q-1)r
\qquad\text{if }j=o(r^{1/3}).
\end{equation}
Thus polynomially superlinear growth cannot occur when
\(j=o(r^{1/3})\).

In the binary case, a classical construction of
Sidel'nikov~\cite{Sidelnikov}
(see also \cite[Remark~4]{Balla}) shows that such growth
already occurs at the scale \(j=\Theta(r^{1/3})\).
More precisely, there is an infinite sequence of positive
integer pairs \((r,j)\) with \(j=\Theta(r^{1/3})\) and
\((r-j)/2\) a positive integer such that
\[
A_2(r,(r-j)/2)=\Omega(r^{4/3}).
\]
By \eqref{eq:code-Ramsey-lower}, the same lower bound holds
for \(R(3;r,(r-j)/2)\).

Conlon, Fox, Pham and Zhao asked whether
polynomially superlinear growth at this scale also occurs for
cliques of order four. We formalize their question as follows.

\begin{question}[Conlon, Fox, Pham and Zhao {\cite{CFPZ}}]
\label{ques:CFPZ}
Does there exist a constant $\varepsilon>0$ such that there are
infinitely many positive integer pairs $(r,j)$ with
$j=\Theta(r^{1/3})$ for which $2(r-j)/3$ is a positive integer and
$R(4;r,2(r-j)/3)\ge r^{1+\varepsilon}$?
\end{question}

We prove the following theorem, which answers
Question~\ref{ques:CFPZ} in the affirmative and extends the binary
sharpness phenomenon to every prime-power alphabet size.

\begin{theorem}\label{thm:main}
Fix a prime power $q$. For each positive integer $m$ with
$Q=q^m>2$, define the positive integers
$r=(Q^2+1)(Q-1)/(q-1)$,
$j=(Q-1)/(q-1)$, and
$s=Q^2(Q-1)/q$.
These parameters satisfy
\[
s=(1-1/q)(r-j)<r,
\qquad
j\sim(q-1)^{-2/3}r^{1/3}
\quad\text{as }m\to\infty.
\]
Along this sequence,
\[
R(q+1;r,s)=\Theta_q(r^{4/3})
\qquad\text{and}\qquad
A_q(r,s)=(1+o(1))(q-1)^{4/3}r^{4/3}.
\]
\end{theorem}

For $q=3$, Theorem~\ref{thm:main} gives
$R(4;r,s)\ge r^{1+\varepsilon}$ along the constructed sequence
for every fixed $0<\varepsilon\le1/3$ and all sufficiently large $r$.
More generally, it shows that the exponent $1/3$ in
\eqref{eq:Balla-code} and \eqref{eq:Balla-Ramsey} is best possible
for every prime power $q$.

Our proof concatenates an ovoid code over $\F_{q^m}$ with a
$q$-ary simplex code to obtain the required lower bounds.
A linear programming bound of Boyvalenkov, Danev and
Stoyanova~\cite{BDS} for codes near the zero-rate threshold
establishes the asymptotic optimality of the constructed code.
The distance reduction in \eqref{eq:CFPZ-upper} preserves
the order of the deficit, allowing the same coding bound
to yield the required Ramsey upper bound.

The construction also gives the same order of growth throughout
local ranges of parameters. To state this extension, write
\[
r_m=\frac{(Q^2+1)(Q-1)}{q-1},
\qquad j_m=\frac{Q-1}{q-1},
\qquad Q=q^m,
\]
and define, for an integer $t$,
\[
\eta_m(t)=
\begin{cases}
j_m+t,&t\ge0,\\[1mm]
j_m-t/(q-1),&t<0.
\end{cases}
\]
The two cases correspond to appending zero coordinates and deleting
coordinates, respectively.

\begin{corollary}\label{cor:local-range}
Fix a prime power \(q\) and constants \(L>0\) and
\(B>(q-1)^{-2/3}+L\).
For all sufficiently large \(m\), suppose that \(r,j\) are positive
integers satisfying
\[
|r-r_m|\le Lr_m^{1/3},
\qquad
\eta_m(r-r_m)\le j\le Br_m^{1/3},
\]
and that \(s=(1-1/q)(r-j)\) is a positive integer.
Then, uniformly over the stated parameters,
\[
A_q(r,s)=\Theta_{q,L,B}(r^{4/3})
\qquad\text{and}\qquad
R(q+1;r,s)=\Theta_{q,L,B}(r^{4/3}).
\]
\end{corollary}

In particular, at $r=r_m$ the conclusion holds throughout
$j_m\le j\le Br_m^{1/3}$, subject to the integrality of $s$.
The corollary gives order-of-growth estimates on these ranges;
the sharper asymptotic formula for $A_q(r,s)$ in
Theorem~\ref{thm:main} concerns the original parameter sequence.

\medskip\noindent
\textbf{Organization.}
Section~\ref{sec:prelim} collects the necessary coding and
finite-field facts.
Section~\ref{sec:proof} proves Theorem~\ref{thm:main},
with an explicit constant in the Ramsey upper bound,
and Corollary~\ref{cor:local-range}.
We mention some related open questions in Section~\ref{sec:cr}.

\section{Preliminaries}\label{sec:prelim}

We begin by fixing notation and recalling some standard facts from coding
theory and finite fields. Let $\F_Q$ denote the finite field of order $Q$.
A \textbf{linear $[n,k,d]_Q$ code} is a $k$-dimensional subspace
$C\subseteq\F_Q^n$ whose minimum Hamming distance is $d$. The code $C$
contains $Q^k$ codewords. The
\textbf{Hamming weight} of a vector is the number of its nonzero
coordinates. For a linear code, the minimum distance equals the minimum
weight of a nonzero codeword.

The ovoid construction will use the quadratic extension of $\F_Q$.
Let $Q>2$ be a prime power, and let $\Tr,\Nm:\F_{Q^2}\to\F_Q$ denote the trace and norm maps, respectively,
given by
\[
\Tr(z)=z+z^Q,\qquad \Nm(z)=z^{Q+1} \qquad (z\in\F_{Q^2}).
\]

The following standard facts about the trace and norm maps will allow us to count the zero coordinates of each codeword.

\begin{fact}[{\cite[Chapter~2, Section~3]{LN}}]
\label{prop:trace-norm}

\medskip
(i) For every $t\in\F_Q$, the equation $\Tr(z)=t$ has exactly $Q$
solutions in $\F_{Q^2}$.

\smallskip
(ii) For every $t\in\F_Q$, the equation $\Nm(z)=t$ has exactly one
solution if $t=0$, and exactly $Q+1$ solutions if $t\ne0$.
\end{fact}

To pass from codes over an extension field to codes over a fixed field,
we shall use the simplex code. Let $q$ be a prime power, let $m\ge1$,
and put $Q=q^m$. The $m$-dimensional $q$-ary simplex code has length
$(Q-1)/(q-1)$ and the following constant-weight property.

\begin{fact}[{\cite[Section~1.8]{HuffmanPless}}]
\label{prop:simplex}
Every nonzero codeword of the $m$-dimensional $q$-ary simplex code has
weight $Q/q$.
\end{fact}

We recall the part of Boyvalenkov, Danev and
Stoyanova's bound that we need. In our distance notation,
\cite[Corollary~4.6, eq.~(13)]{BDS} states that, for a fixed
integer $q\ge2$ and fixed constants $c>0$ and
$\alpha\in[1/5,1/2)$, if
\[
h=cr^\alpha,
\qquad
d=r-1-\frac{r-2+h}{q},
\]
and
\[
0\le h<
\frac{\sqrt{q^2+4(q-1)(r-2)}-q}{2},
\]
then, as $r\to\infty$,
\begin{equation}\label{eq:BDS-original}
A_q(r,d)
\le(q-1)(q+h)r+h^3
+\frac{c^5r^{5\alpha-1}}{q-1}+o(r).
\end{equation}

The following consequence will be used for both coding upper
bounds in the proof of our main theorem.

\begin{lemma}\label{lem:BDS}
Fix an integer $q\ge2$ and a constant $\gamma>0$.
If a positive real sequence $j=j(r)$ satisfies
$j\sim\gamma r^{1/3}$ as $r\to\infty$ through positive integers,
then
\[
A_q\bigl(r,(1-1/q)(r-j)\bigr)
\le(1+o(1))(q-1)^2rj.
\]
\end{lemma}

\begin{proof}
Set
\[
d=\left\lceil(1-1/q)(r-j)\right\rceil,
\qquad
h=(q-1)r-qd-q+2.
\]
Then
\[
d=r-1-\frac{r-2+h}{q},
\qquad
h=(q-1)j+O(1)\sim(q-1)\gamma r^{1/3}.
\]

Fix $\varepsilon>0$, and put
\[
c=(1+\varepsilon)(q-1)\gamma,
\qquad
h_c=cr^{1/3},
\qquad
d_c=r-1-\frac{r-2+h_c}{q}.
\]
For all sufficiently large $r$, we have $h\le h_c$, and hence
$d\ge d_c$. Also, $h_c=O(r^{1/3})$, whereas the upper endpoint
of the admissible range is asymptotic to $\sqrt{(q-1)r}$.
Thus $h_c$ lies in that range for all sufficiently large $r$. By monotonicity in the
distance parameter and \eqref{eq:BDS-original} with $\alpha=1/3$,
\[
\begin{aligned}
A_q(r,d)\le A_q(r,d_c)
&\le(q-1)(q+cr^{1/3})r+c^3r
   +\frac{c^5r^{2/3}}{q-1}+o(r)\\
&=(q-1)c\,r^{4/3}+O(r)\\
&=(1+\varepsilon+o(1))(q-1)^2rj,
\end{aligned}
\]
where the last step uses $j\sim\gamma r^{1/3}$.
Since $\varepsilon>0$ is arbitrary, the desired bound follows.
\end{proof}

\section{Proof of the main theorem}\label{sec:proof}

We first obtain explicit parameters for the concatenated code.
We then apply Lemma~\ref{lem:BDS} to establish its asymptotic
optimality. Finally, we combine the same lemma with the conversion
from set-colorings to codes to obtain the Ramsey upper bound.

\subsection{The ovoid code and the coding lower bound}

We use an explicit algebraic realization of a classical ovoid code;
see, for example, Ding and Heng~\cite{DH}. Fix a prime power $Q>2$.
For $a,c\in\F_Q$ and
$b\in\F_{Q^2}$, define
\[
v(a,b,c)=
\bigl((a\Nm(z)+\Tr(bz)+c)_{z\in\F_{Q^2}},\,a\bigr),
\]
whose first $Q^2$ coordinates are indexed by the elements of $\F_{Q^2}$,
and the final coordinate is $a$.
Set
\[
C_Q=\{v(a,b,c):a,c\in\F_Q,\ b\in\F_{Q^2}\}.
\]

\begin{lemma}
\label{lem:ovoid}
$C_Q$ is a linear $[Q^2+1,4,Q^2-Q]_Q$ code with nonzero
weights $Q^2-Q$ and $Q^2$.
\end{lemma}

\begin{proof}
The parameter map $(a,b,c)\mapsto v(a,b,c)$ is $\F_Q$-linear, and hence
its image $C_Q$ is an $\F_Q$-linear subspace of $\F_Q^{Q^2+1}$.
To determine its dimension and minimum distance, we compute the
weight associated with each nonzero parameter triple.

If $a=0$ and $b\ne0$, multiplication by $b$ permutes $\F_{Q^2}$, so
Fact~\ref{prop:trace-norm}(i) implies that the equation
$\Tr(bz)=-c$ has $Q$ solutions.
The final coordinate is zero, so the weight is $Q^2-Q$.
If $a=b=0$ and $c\ne0$, the weight is $Q^2$.

It remains to consider $a\ne0$. Put $w=b^Q/a$. Since
$a^Q=a$ and $b^{Q^2}=b$, we have $aw=b^Q$ and $aw^Q=b$, and hence
\[
\begin{aligned}
a\Nm(z+w)
&=a(z+w)(z^Q+w^Q)\\
&=a\Nm(z)+aw^Qz+awz^Q+a\Nm(w)\\
&=a\Nm(z)+bz+b^Qz^Q+a\Nm(w).
\end{aligned}
\]
Therefore,
\[
a\Nm(z)+\Tr(bz)+c
=a\Nm(z+w)+c-a\Nm(w).
\]
Translation by $w$ permutes $\F_{Q^2}$. Therefore, by
Fact~\ref{prop:trace-norm}(ii), exactly $1$ or $Q+1$ of the
first $Q^2$ coordinates are zero,
according as $\Nm(w)-c/a$ is zero or nonzero.
Since the final coordinate is nonzero, the corresponding weights
are $Q^2$ and $Q^2-Q$.

Every nonzero parameter triple gives a vector of positive weight,
so the linear parameter map is injective. Its domain
$\F_Q\times\F_{Q^2}\times\F_Q$ has dimension $1+2+1=4$ over $\F_Q$,
and hence $\dim_{\F_Q}C_Q=4$.
Both weights occur already in the cases with $a=0$ considered above.
Since $C_Q$ is linear, its minimum distance is its least nonzero
weight, namely $Q^2-Q$.
\end{proof}

We now concatenate $C_Q$, where $Q=q^m$, with the $m$-dimensional
$q$-ary simplex code: each coordinate over $\F_Q$ is replaced by a
simplex codeword over $\F_q$. The constant nonzero weight in
Fact~\ref{prop:simplex} allows us to compute the resulting minimum
distance exactly.

\begin{lemma}\label{lem:construction}
Let \(q\) be a prime power, let \(m\) be a positive integer, and put
\(Q=q^m>2\). Set
\[
r=\frac{(Q^2+1)(Q-1)}{q-1},\qquad
j=\frac{Q-1}{q-1},\qquad
s=\frac{Q^2(Q-1)}q.
\]
Then \(r,j,s\) are positive integers with \(s<r\),
\(s=(1-1/q)(r-j)\), and there exists a linear \([r,4m,s]_q\) code.
In particular, \(A_q(r,s)\ge Q^4\), and hence
\(R(q+1;r,s)\ge Q^4+1\).
\end{lemma}

\begin{proof}
Let \(C_Q\) be the ovoid code from Lemma~\ref{lem:ovoid}. Using an
\(\F_q\)-linear identification \(\F_Q\cong\F_q^m\), apply a linear simplex
encoder
\[
\sigma:\F_Q\to\F_q^{(Q-1)/(q-1)}
\]
to each coordinate of \(C_Q\), and denote the resulting concatenated code
by \(C\).

Since \(q-1\) divides \(Q-1\) and \(q\) divides \(Q\), the parameters
\(r,j,s\) are positive integers. The outer code \(C_Q\) has \(Q^2+1\)
coordinates, and each is replaced by a simplex codeword of length
\((Q-1)/(q-1)\). Hence the length of \(C\) is
\[
(Q^2+1)\frac{Q-1}{q-1}=r.
\]

Next, \(C_Q\) has dimension four over \(\F_Q\), and \(\sigma\) is
injective and \(\F_q\)-linear. Therefore \(C\) has \(\F_q\)-dimension
\(4m\), and hence contains
\[
q^{4m}=Q^4
\]
codewords.

It remains to determine the minimum distance. By Lemma~\ref{lem:ovoid},
the nonzero weights of \(C_Q\) are \(Q^2-Q\) and \(Q^2\). Since every
nonzero simplex codeword has weight \(Q/q\), an outer word of weight \(w\)
acquires weight \((Q/q)w\). Thus the nonzero weights of \(C\) are
\[
\frac Qq(Q^2-Q)=\frac{Q^2(Q-1)}q=s
\quad\text{and}\quad
\frac Qq Q^2=\frac{Q^3}{q}.
\]
So the minimum distance of \(C\) is \(s\). Moreover,
\[
(1-1/q)(r-j)
=\frac{q-1}{q}\left(
\frac{(Q^2+1)(Q-1)}{q-1}
-\frac{Q-1}{q-1}
\right)
=\frac{Q^2(Q-1)}q=s,
\]
which also shows that \(0<s<r\).

Finally, since \(C\) has \(Q^4\) codewords, we have
\(A_q(r,s)\ge Q^4\). Together with \eqref{eq:code-Ramsey-lower}, this
gives $R(q+1;r,s)\ge A_q(r,s)+1\ge Q^4+1$.
\end{proof}

\subsection{The coding upper bound}

We next show that the constructed code is asymptotically optimal.
The key point is that its deficit parameter lies at the scale required
by Lemma~\ref{lem:BDS}.

\begin{lemma}\label{lem:upper}
Fix a prime power \(q\), and let \(r,j,s\) be as in
Lemma~\ref{lem:construction}. As \(m\to\infty\), we have
\[
A_q(r,s)\le(1+o(1))Q^4.
\]
\end{lemma}

\begin{proof}
The parameters in Lemma~\ref{lem:construction} satisfy
\begin{equation}\label{eq:parameter-asymptotics}
r\sim\frac{Q^3}{q-1},
\qquad j\sim\frac{Q}{q-1},
\qquad
\frac{j}{r^{1/3}}
\sim\frac{Q/(q-1)}{Q/(q-1)^{1/3}}
=(q-1)^{-2/3}.
\end{equation}
Thus $j\sim(q-1)^{-2/3}r^{1/3}$, and Lemma~\ref{lem:BDS} gives
\[
A_q(r,s)=A_q\bigl(r,(1-1/q)(r-j)\bigr)
\le (1+o(1))(q-1)^2rj.
\]
Substituting the exact expressions for \(r\) and \(j\), we obtain
\[
(q-1)^2rj=(Q^2+1)(Q-1)^2=(1+o(1))Q^4.
\]
Therefore \(A_q(r,s)\le(1+o(1))Q^4\), as desired.
\end{proof}

\subsection{Proof of Theorem~\ref{thm:main}}

To obtain the Ramsey upper bound, we use the following result \cite[Theorem~9]{CFPZ}, which converts a lower bound on the Ramsey number into a lower bound on the size of an error-correcting code whose minimum distance is only slightly below the required value.

\begin{lemma}[Conlon, Fox, Pham and Zhao {\cite{CFPZ}}]\label{lem:CFPZ-conv}
Let $q\ge2$ and $r>s\ge1$ be integers, and let $\lambda>1$.
Assume that \(N=R(q+1;r,s)-1\ge12q^2\).
If \(b=\left\lfloor4\lambda q^2\left((1-1/q)r-s+\frac{s}{N}\right)\right\rfloor\ge0\),
then
\[
A_q(r,s-2b)\ge(1-1/\lambda)N.
\]
\end{lemma}

We first prove the coding assertion, and then deduce the Ramsey bounds by combining the construction with Lemma~\ref{lem:CFPZ-conv}.

\begin{proof}[Proof of Theorem~\ref{thm:main}]
Fix a prime power $q$. For each positive integer $m$ with
$Q=q^m>2$, take $r,j,s$ as in Lemma~\ref{lem:construction}.
As $m\to\infty$, the values of $r$ form a strictly increasing
sequence. Lemma~\ref{lem:construction} verifies the asserted
integrality and the identity $s=(1-1/q)(r-j)$, while
\eqref{eq:parameter-asymptotics} gives
$r\sim Q^3/(q-1)$ and
$j\sim(q-1)^{-2/3}r^{1/3}$.
Lemmas~\ref{lem:construction} and~\ref{lem:upper} therefore yield
\[
A_q(r,s)
=(1+o(1))Q^4
=(1+o(1))(q-1)^{4/3}r^{4/3},
\]
which proves the coding assertion.

The construction also gives $R(q+1;r,s)\ge Q^4+1$.
For the upper bound, we apply Lemma~\ref{lem:CFPZ-conv}
with $\lambda=2$ and then estimate the resulting code size.

Put $N=R(q+1;r,s)-1$ and
\[
\Delta=(1-1/q)r-s+1=(1-1/q)j+1.
\]
By Lemma~\ref{lem:construction}, we have $N\ge Q^4$ and
$s=Q^2(Q-1)/q=O(Q^3)$. Hence $N\ge12q^2$ and $s/N\le1$
for all sufficiently large $m$.
Set
\[
b=\left\lfloor
8q^2\left((1-1/q)r-s+\frac{s}{N}\right)
\right\rfloor.
\]
Since $s\le(1-1/q)r$ and $s/N\le1$, we have
\[
0\le b\le8q^2\Delta.
\]
Lemma~\ref{lem:CFPZ-conv} therefore gives
$A_q(r,s-2b)\ge N/2$.
By monotonicity in the distance parameter,
\[
R(q+1;r,s)=N+1
\le2A_q(r,s-2b)+1
\le2A_q(r,s-16q^2\Delta)+1.
\]

To estimate the code size on the right, write
$s-16q^2\Delta=(1-1/q)(r-\widetilde j)$.
Indeed,
\[
\begin{aligned}
s-16q^2\Delta
&=(1-1/q)(r-j)-16q^2\bigl((1-1/q)j+1\bigr)\\
&=(1-1/q)r-(1+16q^2)(1-1/q)j-16q^2\\
&=(1-1/q)\left(r-(1+16q^2)j-\frac{16q^3}{q-1}\right),
\end{aligned}
\]
so we have
\[
\widetilde j
=(1+16q^2)j+\frac{16q^3}{q-1}
\sim(1+16q^2)(q-1)^{-2/3}r^{1/3}.
\]
The adjusted deficit thus remains at the $r^{1/3}$ scale,
so Lemma~\ref{lem:BDS} applies. Using the fact that
$\widetilde j/j\to1+16q^2$ and
$(q-1)^2rj=(Q^2+1)(Q-1)^2\sim Q^4$, we obtain
\[
\begin{aligned}
A_q(r,s-16q^2\Delta)
&\le(1+o(1))(q-1)^2r\widetilde j\\
&=(1+16q^2+o(1))(q-1)^2rj\\
&=(1+16q^2+o(1))Q^4.
\end{aligned}
\]
Together with the construction, this gives
\begin{equation}\label{eq:Ramsey-explicit-bounds}
Q^4+1
\le R(q+1;r,s)
\le\bigl(2(1+16q^2)+o(1)\bigr)Q^4.
\end{equation}
Finally, $Q^4\sim(q-1)^{4/3}r^{4/3}$, and hence
$R(q+1;r,s)=\Theta_q(r^{4/3})$, as required.
\end{proof}

\subsection{Proof of Corollary~\ref{cor:local-range}}

\begin{proof}
Put $t=r-r_m$ and $s_m=(1-1/q)(r_m-j_m)$.
Starting with the code in Lemma~\ref{lem:construction},
append $t$ zero coordinates if $t\ge0$, and delete any $-t$
coordinates if $t<0$. The resulting code has minimum distance
at least
\[
s_m-\max\{-t,0\}
=(1-1/q)\bigl(r-\eta_m(t)\bigr)
\ge (1-1/q)(r-j)=s.
\]
Since $s>0$, all $Q^4$ codewords remain distinct. Thus
\eqref{eq:code-Ramsey-lower} gives
\[
R(q+1;r,s)>A_q(r,s)\ge Q^4
=(1+o(1))(q-1)^{4/3}r^{4/3},
\]
where the last equality holds uniformly because
$r=r_m+O(r_m^{1/3})$.

For the upper bound, we use the same distance-reduction
argument as in the proof of Theorem~\ref{thm:main}.
Apply \eqref{eq:CFPZ-upper} with $\varepsilon=1$, and let
$c=c(q,1)$. With $\Delta=(1-1/q)j+1$, we have
\[
R(q+1;r,s)\le\max\{2A_q(r,s-c\Delta),s\},
\qquad
s-c\Delta=(1-1/q)(r-j'),
\]
where
\[
j'=(1+c)j+\frac{cq}{q-1}.
\]
Since $j\le Br_m^{1/3}$ and $r/r_m\to1$ uniformly,
any fixed $D>(1+c)B$ satisfies $j'\le Dr^{1/3}$
for all sufficiently large $m$. By monotonicity and
Lemma~\ref{lem:BDS},
\[
A_q(r,s-c\Delta)
\le A_q\bigl(r,(1-1/q)(r-Dr^{1/3})\bigr)
=O_{q,B}(r^{4/3}).
\]
Together with $s<r$ and \eqref{eq:code-Ramsey-lower},
this proves both upper bounds, uniformly over the stated range.
\end{proof}

\section{Concluding remarks}\label{sec:cr}

Theorem~\ref{thm:main} determines the maximum code size
asymptotically along an explicit sequence and the corresponding
Ramsey number up to a constant factor. Corollary~\ref{cor:local-range}
shows that the $r^{4/3}$ order of growth persists throughout
local ranges of lengths and deficits. The constant
$2(1+16q^2)$ in \eqref{eq:Ramsey-explicit-bounds} comes from
taking $\lambda=2$ in \cite[Theorem~9]{CFPZ}; we have not
attempted to optimize it.

Conlon, Fox, Pham and Zhao~\cite[Section~V]{CFPZ} discuss how
the growth of set-coloring Ramsey numbers changes as the deficit
increases, and ask whether the transition from linear to
polynomially superlinear growth occurs at the $r^{1/3}$ scale
for cliques of order four. Our results establish this behavior
for every prime power $q$ on the parameter ranges above.
These ranges do not, however, cover all sufficiently large
lengths: for fixed $q$, we have $r_{m+1}/r_m\to q^3$, whereas
the permitted changes in length are only $O(r_m^{1/3})$.
This leaves the following question about the full range of lengths.

\begin{question}\label{ques:all-lengths}
For each fixed prime power \(q\), does there exist a constant
\(C_q>0\) such that, for every fixed \(B>C_q\), the estimate
\(R(q+1;r,(1-1/q)(r-j))=\Theta_{q,B}(r^{4/3})\) holds uniformly
over positive integers \(r,j\) satisfying
\(C_qr^{1/3}\le j\le Br^{1/3}\) and
\((1-1/q)(r-j)\in\mathbb N\), as \(r\to\infty\)?
\end{question}

The upper bound in this question follows from
Lemma~\ref{lem:BDS} and \eqref{eq:CFPZ-upper}, by the same
endpoint argument used in the proof of
Corollary~\ref{cor:local-range}. The issue is therefore to
construct sufficiently large examples for every sufficiently
large length while keeping the deficit at the $r^{1/3}$ scale.

A separate question concerns the leading term of the Ramsey
number. Conlon, Fox, Pham and Zhao~\cite{CFPZ} suggested that
$R(q+1;r,s)=A_q(r,s)+1$ might hold when $s$ is sufficiently
close to $(1-1/q)r$. Their suggestion does not specify a precise
range. Motivated by the transition studied here, we ask for
asymptotic equality throughout the following explicit range.

\begin{question}\label{ques:asymptotic-equality}
Fix a prime power $q$ and a constant $C>0$.
Is it true that
\[
R(q+1;r,s)=(1+o(1))A_q(r,s)
\]
uniformly over positive integers $r,s$ satisfying
\(0\le(1-1/q)r-s\le Cr^{1/3}\) as $r\to\infty$?
\end{question}

Along the sequence in Theorem~\ref{thm:main}, this would give
$R(q+1;r,s)=(1+o(1))Q^4$. On the wider ranges in
Corollary~\ref{cor:local-range}, the question compares the Ramsey
number with $A_q(r,s)$ itself, whose asymptotic leading term
has not been determined there.

\section*{Declaration on the use of AI}

The lower-bound approach in this paper uses families of partitions
of a vertex set into a fixed number of classes. The key difficulty
is to obtain polynomially more vertices than partitions while
ensuring that every pair of vertices is separated by sufficiently
many of the partitions, as required near the zero-rate threshold.

We used ChatGPT (OpenAI) to discuss whether the separation
requirement above had known analogues in coding theory, to assist
with literature searches, and to improve the presentation of the
manuscript. These discussions led us to examine classical
finite-geometric constructions, including ovoid codes and their
concatenation with simplex codes.

The authors reviewed and verified all constructions, parameter
calculations, bounds, and proofs presented in the manuscript.
All cited sources were checked against the original literature.
The authors take full responsibility for the mathematical content
and conclusions of the paper.

\end{document}